\documentclass{article}

\usepackage{amssymb,amsmath,amsthm,verbatim}
\usepackage{enumerate}
\usepackage{xspace}

\newcommand{\GA}{{\rm GA}}
\newcommand{\SG}{\mathfrak{S}}

\newcommand{\BA}{{\rm BA}}
\newcommand{\TA}{{\rm TA}}

\newcommand{\PA}{{\rm PA}}
\newcommand{\GL}{{\rm GL}}
\newcommand{\M}{{\rm M}}
\newcommand{\Af}{{\rm Aff}}
\newcommand{\Tr}{{\rm Tr}}

\DeclareMathOperator{\LND}{LND}

\DeclareMathOperator{\Spec}{Spec}

\newcommand{\A}{\mathbb{A}}

\newcommand{\IN}{\mathbb{N}}

\newcommand{\Ik}{\mathbb{K}}
\newcommand{\K}{\mathbb{K}}

\newcommand{\al}{\alpha}
\newcommand{\bx}{{\bf x}}

\newtheorem{theorem}{Theorem}

\newtheorem{lemma}[theorem]{Lemma}

\newtheorem{corollary}[theorem]{Corollary}

\theoremstyle{definition}
\newtheorem{definition}{Definition}

\theoremstyle{remark}
\newtheorem{remark}{Remark}
\newtheorem{example}{Example}

\newcommand{\supp}{{\ensuremath{\rm supp\ }} }

\newcommand{\cotame}{co-tame\xspace}
\newcommand{\pri}{\ensuremath{\smallsetminus}}
\newcommand{\lex}{\ensuremath{<_{\rm lex}}}
\newcommand{\glex}{\ensuremath{>_{\rm lex}}}

\title{Co-tame polynomial automorphisms}
\author{%
Eric Edo\thanks{ERIM, University of New Caledonia. Email address: \texttt{eric.edo@univ-nc.nc}} \ and %
Drew Lewis\thanks{Department of Mathematics and Statistics,  University of South Alabama.  Email address: \texttt{drewlewis@southalabama.edu}}%
}

\begin{document}
\maketitle

\begin{abstract}
A polynomial automorphism of $\A^n$ over a field of characteristic zero is called {\em \cotame} if, together with the affine subgroup, it generates the entire tame subgroup.   We prove some new classes of automorphisms of $\A^n$, including $3$-triangular automorphisms, are \cotame.  Of particular interest, if $n=3$, we show that the statement ``Every $m$-triangular automorphism is \cotame'' is true if and only if $m \leq 3$; this improves upon positive results of Bodnarchuk (for $m \leq 2$, in any dimension $n$) and negative results of the authors (for $m \geq 6$, $n=3$).  The main technical tool we introduce is a class of maps we term {\em translation degenerate automorphisms}; we show that all of these are \cotame, a result that may be of independent interest in the further study of \cotame automorphisms.
\end{abstract} 

\section{Introduction}

Throughout, we use $\Ik$ to denote a field of characteristic zero, and we denote by $\GA_n(\Ik)$ the group of polynomial automorphisms of $\A^n _\Ik$.  A fundamental question in affine algebraic geometry is to understand the structure of this group and its subgroups, most notably the tame subgroup $\TA_n(\Ik)$, defined as the subgroup generated by affine and triangular automorphisms.    When $n=2$, a classical result of Jung \cite{Jung} is that every automorphism is tame; however, Shestakov and Umirbaev \cite{SU} showed that this is not true when $n=3$.  

In this paper, we are interested in the notion of {\em \cotame automorphisms}, first defined in \cite{E}.  An automorphism is said to be co-tame if, together with the affine subgroup $\Af_n(\Ik)$, it generates the entire tame subgroup.  These automorphisms are particularly interesting when trying to understand the subgroup lattice of $\GA_n(\Ik)$, since any intermediate subgroup between $\Af_n(\Ik)$ and $\TA_n(\Ik)$ must contain an automorphism which is not co-tame.  In dimension one, all automorphisms are affine, and hence all automorphisms are trivially co-tame.  In dimension two, the amalgamated free product structure of $\TA_2(\Ik)$ provides that no automorphisms are co-tame.  However, in dimension three and higher, things become more interesting.  The first example of a co-tame automorphism for $n \geq 3$ was produced by Derksen in 1997 (unpublished), namely $(x_1+x_2^2,x_2,\ldots,x_n) \in \GA_n(\Ik)$.

In 2002, Bodnarchuk \cite{Bodnarchuk02} showed that all triangular maps are either affine or co-tame; further, he showed that all  bitriangular maps (those of the form $\tau _1 \alpha \tau _2$, where $\tau _1,\tau_2$ are triangular and $\alpha$ is affine) are also either affine or \cotame.  He also showed \cite{Bodnarchuk05} that all non-affine parabolic and biparabolic maps are \cotame (see sections \ref{sec:pa} and \ref{sec:para} for precise definitions, and section \ref{sec:known} for proofs of these results).  The first author \cite{E} showed that even certain wild automorphisms, including the famous Nagata map are \cotame.  The first example of an automorphism (in characteristic zero) which is tame but not \cotame was produced by the authors in \cite{EL}.

In this paper, we aim to improve upon the results of Bodnarchuk.  We say a (tame) automorphism is {\em $m$-triangular} if it can be written as $\alpha _0 \tau _1 \alpha _1 \cdots \tau _m \alpha _m$ where each $\tau _i$ is triangular and each $\alpha _i$ is affine.  We remark that a $m$-triangular map could also be $k$-triangular for some $k \neq m$.
Bodnarchuk's results lead us to ask, for which $m$ are all $m$-triangular maps either affine or \cotame?  While the aforementioned results of Bodnarchuk give a positive answer for $m\leq 2$ in all dimensions, the example of~\cite{EL} is a $6$-triangular automorphism in dimension $3$ that is not \cotame.  Our first main result is to improve Bodnarchuk's positive result.

\begin{theorem}\label{mt1} For any $n \geq 3$, every $3$-triangular map is either affine or \cotame.
\end{theorem}

In the final section of this paper we improve the result of the authors in~\cite{EL} to produce a $4$-triangular automorphism that is not \cotame, giving
\begin{theorem}\label{mt2}If $n=3$, the statement ``Every $m$-triangular map is either affine or \cotame'' is true if and only if $m \leq 3$.
\end{theorem}

The basic technique we use to show that a certain class of automorphisms are \cotame is quite simple.  Consider the translation $\theta = \theta_{r,c}:=(x_1,\ldots,x_{r-1},x_r+c,x_{r+1},\ldots,x_n)$ for some $c \in \Ik$; given $\varphi \in \GA_n(\Ik)$, we then consider the map $\widetilde{\varphi} = \varphi ^{-1} \theta \varphi$.  
By carefully choosing the variable $x_r$, we ensure that $\widetilde{\varphi}$ remains in the desired class of maps (e.g. triangular, parabolic, etc.) and set up an induction on an appropriately chosen degree of $\varphi$.
It is easy to see that if $\widetilde{\varphi}$ is co-tame, then $\varphi$ must be as well.  If we can repeat this process and eventually produce a known co-tame map (e.g. Derksen's map, a triangular map, etc.), then $\varphi$ has to be co-tame.

The key obstacle is that this technique breaks down if $\widetilde{\varphi} \in \Af_n(\Ik)$.  
The simple approach is to simply choose a different $c \in \Ik$.  However, for some maps, for a fixed index $r$,  every choice of $c$ results in $\widetilde{\varphi}$ being affine.  We term these maps {\em translation degenerate in $x_r$}, and study them in detail in section \ref{sec:td}.  
In fact, we show the following, which may be of independent interest.

\begin{theorem}\label{mt3} Let $n \geq 3$ and let $\varphi \in \GA_n(\Ik)$.  Fix $1 \leq r \leq n$, and for $c \in \Ik$, set $\theta _{r,c} = (x_1,\ldots,x_{r-1},x_r+c,x_{r+1},\ldots,x_n)$.  If $\varphi ^{-1} \theta _{r,c} \varphi$ is affine for each $c \in \Ik$, then $\varphi$ is either affine or \cotame.
\end{theorem}
  
This result allows us to prove in section \ref{sec:cotame} that various classes of maps, including $3$-triangular maps, are \cotame.  We first give proofs of Bodnarchuk's results in section \ref{sec:known} (including a new proof of the biparabolic case), and prove Theorem \ref{mt1} in section \ref{ss:3tri}.  We also show that any exponential of a multiple of a triangular derivation is \cotame; this class of maps includes the Nagata map.  We reserve the final section for our construction of a $4$-triangular map that is not \cotame; the reader is advised to familiarize themselves with the notations and techniques from \cite{EL} before reading this section, which can be read independently of the rest of this paper.

\subsection{Polynomial Automorphisms}\label{sec:pa}

In this section, $n\ge 1$ is an integer and $R$ is an integral domain.\\

We denote by $R^*$ the group of units of $R$ and  we use $R^{[n]}=R[x_1,\ldots,x_n]$ for the polynomial ring in $n$ variables.
We denote by $\GA_n(R)$ the group of polynomial automorphisms of $\Spec R^{[n]}$ over $\Spec R$.   
This group is anti-isomorphic to the group of $R$-automorphisms of $R^{[n]}$ (some authors define it as such).  We freely abuse this correspondence and, given $\phi\in \GA_n(R)$ and $P \in R^{[n]}$, we denote by $(P)\phi$ the image of $P$ by the automorphism of $R^{[n]}$ corresponding to $\phi$.  Given $\phi, \psi \in \GA_n(R)$ and $P \in R^{[n]}$, we thus have the natural composition $(P)(\phi \psi) = ((P)\phi)\psi$.  We refer the reader to \cite{vdE} for a comprehensive reference on polynomial automorphisms.

The general automorphism group $\GA_n(R)$ has several important subgroups; the following definitions are fairly standard (cf. \cite{Wright2}), but reproduced here for clarity and the convenience of the reader.
\begin{itemize}

\item $\Af_n(R)$ is the affine subgroup, consisting of all automorphisms whose components all have degree one and
$\GL_n(R)\subset \Af_n(R)$ is the linear subgroup, consisting of affine automorphisms whose components have constant term equal to zero.

\item $\BA_n(R)$ is the subgroup of (lower) triangular automorphisms; that is, those of the form $$(u_1x_1+P_1,u_2x_2+P_2(x_1), \ldots, u_nx_n+P_n(x_1,\ldots,x_{n-1}))$$ for some $u_i \in R^*$ and $P_i \in R[x_1, \ldots, x_{i-1}]$.

\item $\TA_n(R) = \langle \Af_n(R), \BA_n(R) \rangle $ is the tame subgroup.  It is a classical theorem of Jung \cite{Jung} that $\TA_2(\Ik) = \GA_2(\Ik)$, while a very deep result of Shestakov and Umirbaev \cite{SU} is that $\TA_3(\Ik) \neq \GA_3(\Ik)$.  The so-called tame generators problem remains open in higher dimensions.

\item $\mathfrak{S}_n$ is the symmetric group on the $n$ symbols $x_1,\ldots,x_n$.
 Given two integers $1\le r,s\le n$ such that $r\ne s$, we denote by 
$(x_r\leftrightarrow x_s)\in\SG_n$ the transposition exchanging $x_r$ and $x_s$.

\item $\Tr_{n,1}(R)$ is the subgroup of translations on the first component, that is the subgroup of automorphisms of the form $\theta_{1,c}:=(x_1+c,x_2,\ldots,x_n)$ for some $c\in R$.

\item $\Tr_{n,r}(R)$ (where $1\le r\le n$) is the subgroup of translations the $r$-th component, that is the subgroup of automorphisms of the form $\theta_{r,c}:=\pi\theta_{1,c}\pi$ for some $c\in R$ where  $\pi=(x_1\leftrightarrow x_r)$. 
In particular, $\Tr_{n,n}(R)$ is the subgroup of translations on the last component, 
that is the subgroup of automorphisms of the form $\theta_{n,c}:=(x_1,\ldots,x_{n-1},x_n+c)$ for some $c\in R$.

\item $\Tr_n(R)$ is the subgroup of all translations, that is the subgroup of automorphisms of the form 
$\theta_{1,c_1}\cdots\theta_{n,c_n}=(x_1+c_1, \ldots,x_n+c_n)$ for some $c_1,\ldots,c_n \in R$.
\end{itemize}

For an $R$-algebra $A$, we use $\LND _R A$ to denote the set of locally  nilpotent $R$-derivations of $A$. 
\begin{itemize}
\item If $D \in \LND _R A$ and $f \in \ker D$, then $fD\in \LND_R A$ as well.
\item Given $D \in \LND _R R^{[n]}$, we denote the exponential of $D$ by $\exp(D) \in \GA_n(R)$, given by $(x_r)(\exp (D)) = \sum _{i=0} \frac{1}{i!} D^i(x_r)$.
\item $D \in \LND_R R^{[n]}$ is called {\em triangular} if $D(x_i) \in R^{[i-1]}$. If $D$ is triangular, then $\exp(D) \in \BA_n(R)$.
\end{itemize}

\subsection{Parabolic Automorphisms}\label{sec:para}

In this section $n\ge 1$ is an integer.

\begin{definition}\label{def:para}
The \textit{parabolic subgroup of} $\GA_n(\K)$ is the normalizer  of $\Tr_{n,n}(\K)$
in $\GA_n(\K)$:
$$\PA_n(\K)=\{\phi\in\GA_n(\K)\,|\,\phi\,\Tr_{n,n}(\K)=\Tr_{n,n}(\K)\,\phi\}.$$
The elements of $\PA_n(\K)$ are called \textit{parabolic automorphisms}.
\end{definition}

In order to prove a characterization of parabolic automorphisms in Lemma \ref{lem:para}, we make the following definition.

\begin{definition}\label{def:partial}
Given $1\le r\le n$, we denote by $\Delta_r:\K^{[n]}\to\K^{[n]}$ the finite partial derivative defined by 
$\Delta_r(P)=(P)\theta_{r,1}-P$ for all $P\in\K^{[n]}$.
\end{definition}

The following lemma is classical and easy to prove via Taylor's theorem (using that $\K$ is a field of characteristic zero).

\begin{lemma}\label{lem:partial}
Given $1\le r\le n$, $a\in\K$ and $P\in\K^{[n]}$, we have:\\
1) $\Delta_r(P)=a\Leftrightarrow P-ax_r\in\K[x_1,\ldots,x_{r-1},x_{r+1},\ldots,x_n]$,\\
2) if $\deg_{x_r}(P)=d\ge 1$ then
$$\Delta_r(P)=\sum_{i=0}^{d-1}\frac{1}{i!}\frac{\partial^iP}{\partial x_r^i}.$$
In particular, $\deg_{x_r}(\Delta_r(P))=d-1$.
\end{lemma}

\begin{lemma}\label{lem:para} 
Let $\phi\in\GA_n(\K)$ be an automorphism. The three following properties are equivalent:\\
$(i)$ $\phi\in\PA_n(\K)$,\\
$(ii)$ $\phi\Tr_{n,n}(\K)\phi^{-1}\subset\Tr_{n,n}(\K)$,\\
$(iii)$ there exist $a\in\K^*$ and $P_1,\ldots,P_n\in\K^{[n-1]}$ such that 
$$\phi=(P_1,\ldots,P_{n-1},ax_n+P_n)=(P_1,\ldots,P_{n-1},x_n)(x_1,\ldots,x_{n-1},ax_n+P_n).$$
\end{lemma}

\begin{proof}\ \\
$(i)\Rightarrow(ii)$: This is obvious.\\
$(ii)\Rightarrow(iii)$: Let $\phi=(Q_1,\ldots,Q_n)\in\GA_n(\K)$ satisfying $(ii)$.
Then there exists $a\in\K^*$ such that $\phi\,\theta_{n,1}\,\phi^{-1}=\theta_{n,a}$.
This implies $\Delta_n(Q_i)=0$ for all $1\le i\le n-1$ and
$\Delta_n(Q_n)=a$. We deduce that $Q_1,\ldots,Q_{n-1},Q_n-ax_n\in\K^{[n-1]}$ using Lemma~\ref{lem:partial}.\\
$(iii)\Rightarrow(i)$: If $\phi$ is as in $(iii)$ then
$\phi\,\theta_{n,c}\,\phi^{-1}=\theta_{n,ac}$ and $\phi^{-1}\,\theta_{n,c}\,\phi=\theta_{n,c/a}$ for all $c\in\Ik$
and we deduce that $\phi$ is a parabolic automorphism.
\end{proof}

\begin{remark}
Lemma~\ref{lem:para} implies that $\PA_n(\K)$ is the semi-direct product of the subgroup
$$\{(P_1,\ldots,P_{n-1},x_n)\in\GA_n(\K)\,|\,P_1\ldots,P_{n-1}\in\K^{[n-1]}\},$$ which we
identify with $\GA_{n-1}(\K)$, and $$\{(x_1,\ldots,x_{n-1},ax_n+P_n)\in\GA_n(\K)\,|\,a\in\K^*,P_n\in\K^{[n-1]}\},$$
which is a subgroup of $\BA_n(\K)$ that we identify with $\BA_1(\K^{[n-1]})$. 
\end{remark}

\begin{remark}
When $1\le n\le 3$, we have $\PA_n(\K)\subset\TA_n(\K)$. 
Nevertheless, we don't know if all parabolic automorphisms are tame when $n\ge 4$.
\end{remark}

\section{Families of cotame automorphisms}\label{sec:cotame}
Throughout this section, $n\ge 3$ is an integer.\\

In this section, we prove that various classes of maps are \cotame, including $3$-triangular maps in section \ref{ss:3tri}.  Some of these proofs rely on Theorem \ref{thm:tdcotame}, which we prove in section \ref{sec:td}.  

\begin{definition}\label{def:cotame}\ \\
1) An automorphism $\phi\in\GA_n(\Ik)$ is said to be \textit{co-tame} if $\TA_n(\Ik)\subset\langle f,\Af_n(\Ik)\rangle$.\\
2) Given $\phi,\psi\in\GA_n(\Ik)$, we write $\phi\simeq\psi$ when $\phi$ and $\psi$ are \textit{affinely equivalent},
i.e., if there exist $\alpha_1,\alpha_2\in\Af_n(\Ik)$ such that $\phi=\alpha_1\psi\alpha_2$.
It follows immediately that if $\phi\simeq\psi$, then  $\phi$ is co-tame if and only if $\psi$ is co-tame.
\end{definition}

\subsection{Known results}\label{sec:known}

In this section, we briefly discuss the known results on families of automorphisms that are \cotame (the reader familiar with this history may skip ahead to section \ref{ss:3tri}).  The first example of a \cotame automorphism was produced by Derksen (unpublished, but see \cite{vdE} Theorem 5.2.1 for a proof).
\begin{theorem}[Derksen, 1997]\label{thm:derksen}
The automorphism $(x_1+x_2^2,x_2,\ldots,x_n)\in\BA_n(\Ik)$ is  \cotame.
\end{theorem}

Subsequently, Bodnarchuk showed that all triangular and parabolic automorphisms are either affine or \cotame; we are unaware of proofs appearing in English in the literature, so we include proofs of these two statements here.
\begin{theorem}[Bodnarchuk, 2002]\label{thm:triangular}
A triangular automorphism $\tau \in \BA_n(\Ik)$ is either affine or \cotame.
\end{theorem}

\begin{proof}
Note that without loss of generality, we may assume that the affine part of $\tau$ is the identity
(because $\tau$ is affinely equivalent to an automorphism of this form).  
We assume that $\tau$ is not affine and we write $(x_i)\tau=x_i+P_i(x_1,\ldots,x_{i-1})$ 
with $P_i\in\K^{[i-1]}$ for each $1 \leq i \leq n$.  
First, suppose $\deg P_r \ge 2$ for some $1 \leq r \leq n-1$.  
Then we consider $\alpha = (x_1,\ldots,x_{n-1},x_n+x_r) \in \Af_n(\Ik)$ 
and we observe that $\tau ^{-1} \alpha \tau = (x_1,\ldots,x_{n-1},x_n+P_r)$.

So we may now assume that $\tau = (x_1,\ldots,x_{n-1},x_n+P)$ for some $P \in \Ik^{[n-1]}$ with $\deg P \geq 2$
(since $\tau^{-1}\alpha\tau$ being \cotame implies $\tau$ is \cotame).
Now, let $1 \leq s \leq n-1$ be such that $P \in \Ik^{[s]} \pri\Ik^{[s-1]}$.  
We assume $s>1$, and will return to the $s=1$ case (i.e., $P \in \Ik[x_1]\pri \Ik$) momentarily.  
Set $\delta  = (x_1,\ldots, x_{s-1},x_s+x_1,x_{s+1},\ldots,x_n) \in \Af_n(\Ik)$.  
Then one easily computes that 
$\delta^{-1}\tau^{-1}\delta\tau =(x_1,\ldots,x_{n-1},x_n+P_0)$ 
where $P_0\in \Ik^{[n-1]}$ with $\deg P_0 = \deg P \ge 2$, and $\deg _{x_s} P_0 < \deg _{x_s} P$.  
Inducting downwards on $\deg _{x_s} P$, we may assume $P \in \Ik^{[s-1]}$ with $\deg P \ge 2$.  
Inducting downwards on $s$, we may assume further that $P \in \Ik[x_1] \pri \Ik$.  

Now, let $\theta =\theta_{1,1}=(x_1+1,x_2,\ldots,x_n)\in\Tr_{n,1}(\Ik)$, 
and note that $\theta ^{-1} \tau ^{-1} \theta \tau = (x_1,\ldots,x_{n-1},x_n+P_0)$ 
with $P_0 =-\Delta_1(P) \in \Ik[x_1]$ and $\deg P_0 = \deg P - 1$.  Inducting downwards again on $\deg P$, we may assume $\deg P = 2$, in which case $\tau$ is affinely equivalent to Derksen's map $(x_1+x_2^2,x_2,\ldots,x_n)$, and is therefore co-tame by Theorem~\ref{thm:derksen}.
\end{proof}

\begin{theorem}[Bodnarchuk, 2002]\label{thm:parabolic}
A parabolic automorphism $\psi \in \PA_{n}(\Ik)$ is either affine or \cotame.
\end{theorem}

\begin{proof}
Using Lemma~\ref{lem:para}, we write $\psi=(P_1,\ldots,P_{n-1},ax_n+P_n)$ 
with $a\in\K^*$ and $P_1,\ldots,P_n\in\K^{[n-1]}$.
If  $(P_1,\ldots,P_{n-1})\in \Af_{n-1}(\Ik)$, then $\psi$ is affinely equivalent to
$(x_1,\ldots,x_{n-1},ax_n+P_n)\in\BA_n(\K)$, 
so $\psi$ is either affine or co-tame by Theorem~\ref{thm:triangular}.  
Otherwise, there exists $1 \leq r \leq n-1$ with $\deg(P_r)\ge 2$.  
We consider $\alpha = (x_1,\ldots,x_{n-1},x_n+ax_r) \in \Af_n(\Ik)$ and we have: 
$$\psi ^{-1} \alpha \psi =(x_1,\ldots,x_{n-1},x_n+P_r) \in \BA_n(\Ik)\pri\Af_n(\K).$$  
Again Theorem \ref{thm:triangular} implies $\psi ^{-1} \alpha \psi$ and then $\psi$ are co-tame.
\end{proof}
\begin{remark}
This immediately implies that every automorphism is {\em stably \cotame}, as the addition of a variable makes the map parabolic.  The question of stable co-tameness is more intricate over a field of positive characteristic; see \cite{K}.
\end{remark}

Bodnarchuk later showed that all biparabolic maps are \cotame; we present a new proof using the fact that translation degenerate maps (see Definition \ref{def:td}) are either affine or \cotame (Theorem \ref{thm:tdcotame}). 
\begin{theorem}[Bodnarchuk, 2005] \label{thm:biparabolic} 
Let $\psi _1, \psi _2 \in \PA_n(\Ik)$ and let $\alpha \in \Af_n(\Ik)$.  
Then $\psi _1 \alpha \psi _2$ is either affine or \cotame.
\end{theorem}

\begin{proof}
If $\psi _1 \alpha \psi _2$ is translation degenerate in $x_n$ then $\psi _1 \alpha \psi _2$ is  either affine or \cotame by
Theorem~\ref{thm:tdcotame}.
Suppose $\psi _1 \alpha \psi _2$ is not translation degenerate in $x_n$.  
Thus we may choose $c \in \Ik$ such that, setting $\theta=\theta_{n,c}$, 
$(\psi _1 \alpha \psi _2)^{-1} \theta (\psi _1 \alpha \psi _2)$ is not affine.  We compute
$$(\psi _1 \alpha \psi _2)^{-1} \theta (\psi _1 \alpha \psi _2) = \psi _2 ^{-1} \alpha ^{-1} \theta_{n,c/a} \alpha \psi _2$$
where $a\in\K^*$ is such that $(x_n)\psi_1-ax_n\in\K^{[n-1]}$.
Note that $\alpha ^{-1} \theta_{n,c/a}\alpha \in \Tr_n(\Ik)$, and therefore $\psi _2 ^{-1} \alpha ^{-1} \theta_{n,c/a} \alpha \psi _2 \in \PA_n(\Ik)$.  Since it is nonaffine, it must be \cotame by Theorem \ref{thm:parabolic}, and thus $\psi _1 \alpha \psi _2$ is \cotame.
\end{proof}

\subsection{3-triangular maps are \cotame}\label{ss:3tri}

In this section we prove Theorem \ref{mt1}.  Suppose $\phi$ is a 3-triangular map, i.e., 
$\phi=\alpha _0 \tau _1 \alpha _1 \tau _2 \alpha _2 \tau _3 \alpha _3$  for some $\tau _i \in \BA_n(\Ik)$, 
and $\alpha _i \in \Af_n(\Ik)$.  
First, since $\Af_n(\Ik)=\GL_n(\Ik) \ltimes \Tr_n(\Ik)$, it suffices to assume $\alpha _i \in \GL_n(\Ik)$.  
Next, observe that we can can assume $\alpha _0 = \alpha _3 = {\rm id}$ (since $\phi$ is affinely equivalent to 
$\alpha _0 ^{-1} \phi \alpha _3^{-1}$).
If $\phi$ is translation degenerate in $x_n$ then $\phi$ is  either affine or \cotame by
Theorem~\ref{thm:tdcotame}, so we suppose $\phi$ is not translation degenerate in $x_n$.  
Thus we may choose $c \in \Ik$ such that, setting $\theta=\theta_{n,c}$, $\phi ^{-1} \theta \phi$ is not affine.  
Then it suffices to show $\phi ^{-1} \theta \phi$ is \cotame; we compute
$$\phi ^{-1} \theta \phi = \tau _3 ^{-1} \alpha _2 ^{-1} \left(\tau _2 ^{-1} \alpha _1 ^{-1} \tau _1 ^{-1} \theta \tau _1 \alpha _1 \tau _2\right) \alpha _2 \tau _3.  $$

However, note that $\tau _1 ^{-1} \theta \tau _1 = \theta$, and thus $\alpha _1 ^{-1} \tau _1 ^{-1} \theta \tau _1 \alpha _1 = \alpha _1 ^{-1}  \theta  \alpha _1 \in \Tr_n(\Ik)$, which implies $\tau _2 ^{-1}\left( \alpha _1 ^{-1} \tau _1 ^{-1} \theta \tau _1 \alpha _1 \right)\tau _2 \in \BA_n(\Ik)$.  Thus, Theorem \ref{mt1} is a consequence of the following theorem.

\begin{theorem}\label{thm:3tri}
Let $\psi \in \PA_n(\Ik)$, let $\tau \in \BA_n(\Ik)$, and let $\alpha \in \GL_n(\Ik)$.  Then $\psi ^{-1} \alpha ^{-1} \tau \alpha \psi$ is either affine or \cotame.
\end{theorem}

Before proving this, we require a definition and a lemma that will provide the crucical inductive step in the proof of Theorem \ref{thm:3tri}. 

\begin{definition}
For $1 \leq r \leq n$, we define the map $d_r : \BA_n(\K) \rightarrow \IN^n$ by 
$$d_r( \tau) = \left( \deg _{x_r} ((x_1)\tau), \ldots , \deg _{x_r} ((x_n)\tau) \right).$$
\end{definition}

We will use ${\bf e}_r$ to denote the $r$-th standard basis vector of $\IN^n$, and $\lex$ to denote the lexicographic 
ordering on $\IN^n$.  Note that for any $\tau \in \BA_n(\Ik)$, if we write $(x_i)\tau = u_ix_i+P_i$ for some $u_i \in \Ik^*$ and $P_i \in \Ik^{[i-1]}$, then we have $$d_r(\tau) = \left(0,\ldots,0,1,\deg _{x_r} (P_{r+1}),\ldots,\deg_{x_r}(P_n) \right) \geq _{\rm lex} {\bf e}_r.$$

\begin{lemma}\label{lem:drdecrease}
Let $\tau \in \BA_n (\Ik)$ and let 
$\theta = \theta_{r,c_r}\ldots\theta_{n,c_n}=(x_1,\ldots,x_{r-1},x_r+c_r,,\ldots,x_n+c_n) \in \Tr_n(\Ik)$ 
for some $c_r \in \Ik^*$ and $c_{r+1},\ldots,c_n \in \Ik$.  
Set $\tilde{\tau} =  \tau ^{-1} \theta  \tau \in \BA_n (\Ik)$.  
If $d_r(\tau) \glex {\bf e}_r$, then  $d_r(\tilde{\tau}) \lex d_r(\tau)$. 
\end{lemma}

\begin{proof}
Write $\tau = \tau _1 \cdots  \tau _n$, where $\tau _i = (x_1,\ldots,x_{i-1},u_ix_i+P_i,x_{i+1},\ldots,x_n)$ for some $u_i \in \Ik^*$ and $P_i \in \Ik^{[i-1]}$ ($1\le i\le n$).  

Observe that
\begin{align*}
\tilde{\tau} &=  \tau _n ^{-1} \cdots \tau _1 ^{-1} \theta \tau _1 \cdots \tau _n \\
&= \tau _n ^{-1} \cdots \tau _{r} ^{-1} \theta \tau _{r} \cdots \tau _n.
\end{align*}

Write $\tilde{\tau} = (x_1,\ldots,x_{r-1},x_{r}+Q_{r},\ldots,x_n+Q_n)$ for some $Q_i \in\Ik[x_1,\ldots,x_{i-1}]$.   
It is easy to see that, for each $r \leq i \leq n$, 
$$\tau _i ^{-1} \cdots \tau _{r} ^{-1} \theta _r \tau _{r} \cdots \tau _i = (x_1,\ldots,x_{r-1},x_r+Q_r,\ldots,x_i+Q_i,x_{i+1},\ldots,x_n),$$
which immediately implies that $Q_r = u_r^{-1}c_r$ and for $r+1 \leq i \leq n$,
$$Q_i = P_i(x_1,\ldots,x_{i-1})-P_i(x_1,\ldots,x_{r-1},x_r+Q_r,\ldots,x_{i-1}+Q_{i-1}).$$

Since $d_r(\tau) \glex {\bf e}_r$, we may let  $l \geq r+1$ be minimal such that $\deg _{x_r}(P_l) > 0$.  Note that this implies   $Q_i \in \Ik[x_1,\ldots,x_{r-1},x_{r+1},\ldots,x_{i-1}]$ for $r<i<l$; then one easily sees from Taylor's formula that $ \deg _{x_{r}} (Q_l) = \deg _{x_{r}}(P_l)-1$, and thus $d_r(\tilde{\tau}) \lex d_r (\tau)$.
\end{proof}

\begin{proof}[Proof of Theorem \ref{thm:3tri}]
Set $\phi = \psi ^{-1} \alpha ^{-1} \tau \alpha \psi$.  Without loss of generality, we may assume that $\phi$ is not translation degenerate.  Then there exists $c \in \Ik$ such that (letting $\theta =\theta_{n,c}$)  $\phi ^{-1} \theta \phi$ is not affine.  Set
\begin{align*}
\tilde{\theta} &= \alpha \theta \alpha ^{-1} &\text{and}& &\tilde{\tau} &= \tau ^{-1}\tilde{\theta} \tau, 
\end{align*}
and note $\tilde{\theta} \in \Tr_n(\Ik)$.  In particular, writing $$\alpha = (a_{1,1}x_1+\cdots+a_{1,n}x_n,\ldots,a_{n,1}x_1+\cdots+a_{n,n}x_n) \in \GL_n(\Ik),$$ we have $\tilde{\theta}=(x_1+a_{1,n}c,\ldots,x_n+a_{n,n}c)$.  

We compute
\begin{align*}
\phi ^{-1} \theta \phi &= \psi ^{-1} \alpha ^{-1} \tau ^{-1} \alpha \psi \theta \psi ^{-1} \alpha ^{-1} \tau \alpha \psi \\
&= \psi ^{-1} \alpha ^{-1} \tau ^{-1} \alpha \theta  \alpha ^{-1} \tau \alpha \psi \\
&= \psi ^{-1} \alpha ^{-1} \tau ^{-1}\tilde{\theta} \tau \alpha \psi \\
&= \psi ^{-1} \alpha ^{-1} \tilde{\tau} \alpha \psi .
\end{align*}

First, let us suppose that $\tilde{\tau} \in \Af_n(\Ik)$. Then $\alpha ^{-1} \tilde{\tau} \alpha \in \Af_n(\Ik)$, in which case $\phi ^{-1} \theta \phi = \psi ^{-1} \left(\alpha ^{-1} \tilde{\tau} \alpha \right) \psi $ is nonaffine and biparabolic, and thus \cotame by Theorem \ref{thm:biparabolic}.

Next, we suppose instead that there exist some $1 \leq r \leq n$ with $a _{r,n} \neq 0$ and $d_r(\tau) \glex {\bf e}_r$; 
we claim that if no such $r$ exists, then $\tilde{\tau} \in \Af_n(\Ik)$.  
Indeed, write $(x_i)\tau = u_ix_i+P_i$ for some $u_i \in \Ik$ and $P_i \in \Ik[x_1,\ldots,x_{i-1}]$ for each $1 \leq i \leq n$.  Then $(P_i) \tilde{\theta}=P_i$ for each $i$, and a simple computation shows $(x_i) \tilde{\tau} = x_i+u_i^{-1}a_{i,n}c$ for each $1 \leq i \leq n$, in which case $\tilde{\tau} \in \Tr_n(\Ik)$.

So we let $r$ be minimal such that $a _{r,n} \neq 0$ and $d_r(\tau) \glex {\bf e}_r$.
By Lemma \ref{lem:drdecrease}, we have $d_r(\tilde{\tau}) \lex d_r(\tau)$.  We can thus repeat the process until we obtain $d_r(\tilde{\tau})={\bf e}_r$, allowing us to induct upwards on $r$.  This process must eventually terminate by producing either a translation degenerate map (which is \cotame by Theorem \ref{thm:tdcotame}), or a map of the form $\psi ^{-1} \alpha \tilde{\tau} \alpha \psi$ in which $\tilde{\tau}$ is affine, which we showed above must be \cotame.
\end{proof}

\subsection{Some results on exponentials}\label{sec:exponentials}

In \cite{E}, the first author showed that the Nagata map, shown to be wild by Shestakov and Umirbaev \cite{SU}, is \cotame.  Since it is well known that the Nagata map can be given as the exponential of a multiple of triangular locally nilpotent derivation, we can view the following as a generalization of this result (and also as a generalization of Bodnarchuk's result that triangular automorphisms are \cotame).
\begin{theorem}\label{thm:expcotame}
Let $D \in \LND _\Ik \Ik^{[n]}$ be triangular, and let $F \in \ker D$.  Then $\exp(FD)$ is either affine or co-tame.
\end{theorem}

Before proving this, we remark that this is in striking parallel to the well known result of Smith \cite{Smith} that all such maps are stably tame.
\begin{theorem}[M. Smith]
Let $D \in \LND _\Ik \Ik^{[n]}$ be triangular, and let $F \in \ker D$.  Then $\exp(FD)$ is stably tame.
\end{theorem}

In fact, rather than prove Theorem \ref{thm:expcotame}, we generalize to the following.

\begin{theorem}\label{thm:2para1exp} Let $D \in \LND _\Ik \Ik^{[n]}$ be triangular and let $F \in \ker D$.  Let $\psi _1, \psi _2 \in \PA_n(\Ik)$.  Then $\phi = \psi _1  \exp(FD) \psi _2$ is either affine or cotame.
\end{theorem}

\begin{proof}
First, observe that if $F \in \Ik^{[n-1]}$, then $\exp(FD)$ and then $\phi$ is parabolic and thus either affine or \cotame.
We therefore induct downward on $\deg _{x_n} F$; assume $\deg _{x_n} F > 0$.  
By Theorem~\ref{thm:tdcotame}, we may assume that $\exp(FD)$ is not translation degenerate.  Then there exists $c \in \Ik$ such that, letting $\theta =\theta_{n,c}$, we have $\theta ^{-1} \exp(-FD) \theta \exp(FD) \notin \Af_n(\Ik)$.  
But observe that 
$$\theta ^{-1} \exp(-FD) \theta \exp(FD) = \exp( (F-(F)\theta )D),$$ 
and $\deg _{x_n} (F-(F)\theta) < \deg _{x_n} F$ by Taylor's theorem. Finally, compute
\begin{align*}
\theta ^{-1}(\psi _1 \exp(FD) \psi _2)^{-1} \theta (\psi _1 \exp(FD) \psi _2)
&= \psi _2 ^{-1} \left(\theta ^{-1} \exp(-FD) \theta \exp(FD) \right) \psi _2 \\
&= \psi _2 ^{-1} \exp \left(\left( F-(F)\theta \right)D \right) \psi _2.
\end{align*}
Induction completes the proof.
\end{proof}

\section{Translation degenerate automorphisms}\label{sec:td}

In this section, $n\ge 3$ is an integer.\\

We now aim to understand the key technical hurdle appearing in the proofs of the results in the previous section, which we term {\em translation degenerate automorphisms}. The goal of this section is to prove that they, too, are either affine or co-tame (Theorem \ref{thm:tdcotame}).

\begin{definition}\label{def:td}
Given $1 \leq r \leq n$,
an automorphism $\phi\in\GA_n(\Ik)$ is called {\em translation degenerate in $x_r$} 
if $\phi^{-1}\,\Tr_{n,r}(\Ik)\,\phi\subset \Af_n(\Ik)$.
For the sake of brevity, we will say $\phi$ is {\em translation degenerate} 
if it is translation degenerate in at least one $x_r$.
\end{definition}

\begin{example}
If $\phi \in \PA_{n}(\Ik)$, then $\phi$ is translation degenerate in $x_n$.
\end{example}  

\begin{remark}
Fix $1 \leq r \leq n$. Let $\phi\in\GA_n(\Ik)$ be an automorphism.\\
a) For all $\alpha \in \Af_n(\Ik)$, $\phi$ is translation degenerate in $x_r$ if  and only if $\phi\alpha$ is.\\
b) $\phi$ is translation degenerate in $x_r$ if and only if $\pi\phi$ is translation degenerate in $x_1$,
where  $\pi=(x_1\leftrightarrow x_r)$. So we focus on the case $r=1$.
\end{remark}

\begin{remark}
Let $\tau\in \BA_3(\Ik)$ be a triangular automorphism. 
If $\tau$ is translation degenerate in $x_1$, then it need not be the case that $\tau^{-1}$ 
is also translation degenerate in $x_1$ (as one can see from Example \ref{ex:ttd} below).  
\end{remark}

\begin{example}\label{ex:ttd}
Let $\tau = (x_1,x_2-\frac{1}{2}x_1^2,x_3-x_1x_2+\frac{1}{3}x_1^3) \in \BA_3(\Ik)$.  
We have $\tau^{-1} = (x_1,x_2+\frac{1}{2}x_1^2,x_3+x_1x_2+\frac{1}{6}x_1^3)$.
For any $c \in \Ik^*$, one easily computes 
$$\tau^{-1}\theta_{1,c}\tau = \left(x_1+c, x_2+cx_1+\frac{1}{2}c^2,x_3+cx_2+\frac{1}{2}c^2x_1+\frac{1}{6}c^3\right)
\in\Af_3(\K)\ {\rm and}$$
$$\tau\theta_{1,c}\tau^{-1} = \left(x_1+c, x_2-cx_1-\frac{1}{2}c^2,x_3-cx_2+\frac{1}{2}cx_1^2+c^2x_1+
\frac{1}{3}c^3\right)\not\in\Af_3(\K).$$
Therefore $\tau$ is translation degenerate in $x_1$ but $\tau^{-1}$ is not.
\end{example}

As is often the case, it is useful to think of the group $\Tr_{n,r}(\K)$ in Definition~\ref{def:td} 
as an element of $\Af_n(\Ik[t])$, and specialize by substituting $c \in \Ik$ in for $t$.  
We first check that this technique behaves well with respect to our definition.

\begin{lemma} 
Let $\alpha \in \GA_n(\Ik[t])$.  Let $\alpha_c\in\GA_n(\Ik)$ denote the image of $\alpha$ under the specialization map 
$\GA_n(\Ik[t])\rightarrow\GA_n(\Ik)$ given by $t\mapsto c$.  
If $\alpha_c\in\Af_n(\Ik)$ for every $c\in\Ik$, then $\alpha\in\Af_n(\Ik[t])$.
\end{lemma}

\begin{proof}
For each $1 \leq i \leq n$, write $\alpha (x_i) = \sum _{v \in \IN^n} a _{i,v}(t) x^v $ for some $a _{i,v} \in \Ik[t]$.  If $a _{i,v}(t)$ is nonzero for some monomial $x^v$ of degree greater than 1, we can choose $c \in \Ik$ such that $a_{i,v}(c) \neq 0$, contradicting that $\alpha _c \in \Af_n(\Ik)$.  Thus we must have $\alpha \in \Af_n(\Ik[t])$.
\end{proof}

\begin{corollary}\label{cor:affinet}
Let $\phi\in\GA_n(\Ik)$.  Then $\phi$ is translation degenerate in $x_1$ if and only if  $\phi^{-1}\theta_{1,t}\phi\in\Af_n(\Ik[t])$.
\end{corollary}

Now, we proceed with trying to describe translation degenerate maps in sufficient detail.  Our goal is to prove a factorization theorem, Theorem \ref{thm:tdfactor}, which will allow us to prove in 
Theorem~\ref{thm:tdcotame} that all translation degenerate maps are either affine or \cotame.  We begin with two lemmas.
The first idea is to use Taylor's theorem to interpret the assumption \emph{$\phi$ is translation degenerate in $x_1$}
as a differential equation and then solve this equation.

\begin{lemma} \label{lem:nilpotent}
Let $\phi \in \GA_n(\Ik)$ be translation degenerate in $x_1$.  
Then, writing $\phi^{-1} = (H_1,\ldots,H_n)$, there exist $a _{i,j,k}, b _{i,k} \in \Ik$ such that $$\frac{\partial ^k H_i}{\partial x_1 ^k} =  \sum _{j=1} ^n a _{i,j,k} H_j + b _{i,k}$$ for each $1 \leq i \leq n$ and $k \geq 0$.  Moreover, letting $A=(a_{i,j,1}) \in \M_n(\Ik)$ and $B=(b _{j,1}) \in \M_{n,1}(\Ik)$, $A$ is nilpotent; and for each $k>1$, we have 
\begin{align*}
a _{i,j,k} &=  (A^k)_{i,j} \\
b _{i,k} &=  (A^{k-1} B )_{i,1}.
\end{align*}
\end{lemma}

\begin{proof}
Write $\phi ^{-1} = (H_1,\ldots, H_n)$, and let $\theta =\theta_{1,t}\in \Af_n(\Ik[t])$, so we have
$$(x_i) \phi ^{-1} \theta = H_i(x_1+t,x_2,\ldots,x_n) = \sum _{k=0} \frac{t^k}{k!} \frac{\partial ^k H_i}{\partial x_1 ^k}.$$
By Corollary \ref{cor:affinet}, we may write $\phi ^{-1} \theta = \alpha \phi ^{-1}$ for some $\alpha \in \Af_n(\Ik[t])$.  For each $1 \leq i \leq n$, write $ (x_i)\alpha = \sum _{j=1} ^n a _{i,j}(t)x_j + b _i(t)$ for some $a _{i,j}(t), b _i(t) \in \Ik[t]$.  Then we have for each $1 \leq i \leq n$

$$\sum _{k=0} \frac{t^k}{k!} \frac{\partial ^k H_i}{\partial x_1 ^k}  = \sum _{j=1} ^n a _{i,j}(t) H_j + b _i(t).$$
Write $a _{i,j}(t) = \sum _{k=0} a _{i,j,k} \frac{ t^k}{k!}$ and $b_i (t) = \sum _{k=0} b _{i,k} \frac{t^k}{k!}$.  Then we see 
$$\sum _{k=0} \frac{t^k}{k!} \frac{\partial ^k H_i}{\partial x_1 ^k}  = \sum _{k=0} \frac{ t^k}{k!} \left(\sum _{j=1} ^n a _{i,j,k} H_j + b _{i,k}\right).$$
So in particular, we have for each $k \geq 0$, $$ \frac{\partial ^k H_i}{\partial x_1 ^k}  = \sum _{j=1} ^n a _{i,j,k} H_j + b _{i,k}.$$

We prove the formula for $a _{i,j,k}$ and $\beta _{i,k}$ by induction on $k$.  We compute
\begin{align*}
\frac{\partial ^k H_i}{\partial x_1 ^k} &= \frac{\partial}{\partial x_1} \left( \sum _{j=1} ^n a _{i,j,k-1} H_j + b _{i,k} \right) \\
&= \sum _{j=1} ^n a _{i,j,k-1} \left( \sum _{l=1} ^n a _{j,l,1}H_l+b _{j,1}\right) \\
&= \sum _{l=1} ^n H_l \left( \sum _{j=1} ^n (A^{k-1})_{i,j} a _{j,l,1} \right) + \sum _{j=1} ^n  (A^{k-1})_{i,j} b _{j,1} \\
&= \sum _{l=1} ^n H_l (A^k)_{i,l} + (A^{k-1}B)_{i,1}.
\end{align*}
Finally, we observe that this formula immediately implies that $A=(a_{i,j,1})$ is nilpotent, as $a_{i,1,k}=\cdots=a_{i,n,k}=0$ whenever $k>\deg_{x_1} H_i$.
\end{proof}

\begin{lemma}\label{lem:nilpotentJordan}
Let $\phi \in \GA_n(\Ik)$ be translation degenerate in $x_1$.  
There exists $\lambda \in \GL_n(\Ik)$ such that writing $(\phi \lambda)^{-1} = (H_1,\ldots,H_n)$ and  $\frac{\partial ^k H_i}{\partial x_1 ^k} = \sum _{j=1} ^n a _{i,j,k} H_j + b _{i,k}$ as in Lemma \ref{lem:nilpotent}, we have  $(a _{i,j,1}) \in M_n(\Ik)$ is a nilpotent lower triangular Jordan matrix with $b _{1,1}\neq 0$.  
\end{lemma}

\begin{proof}
As in Lemma \ref{lem:nilpotent}, let $\theta = \theta_{1,t} \in \Af_n(\Ik[t])$, and write $\phi ^{-1} \theta = \alpha \phi ^{-1}$ for some $\alpha \in \Af_n(\Ik[t])$.  Note that for any $\lambda \in \GL_n(\Ik)$ we have $(\phi \lambda) ^{-1} \theta = \lambda ^{-1} \alpha  \lambda ( \phi \lambda)^{-1}$.  Thus, simply choose $\lambda \in \GL_n(\Ik)$ such that $\lambda ^{-1} (a _{i,j,1}) \lambda$ is in lower triangular Jordan form.

To see that we can also take $b _{1,1} \neq 0$, note that by expanding along the $r$-th column, we see that the the unital Jacobian determinant $J( H_1,\ldots,H_n)$ is contained in the ideal $\left( b _{1,1}, b _{1,2} + a _{2,1,1}H_1, \ldots, b _{1,n}+a _{n,n-1,1} H_{n-1}\right)$.  Choose a point $(c_1,\ldots,c_n) \in \Ik^n$ such that $H_i (c_1,\ldots,c_n) = -\frac{b _{1,i+1}}{a_{i+1,i}}$ whenever $a _{i+1,i} \neq 0$.  Going modulo the ideal $(x_1-c_1,\ldots,x_n-c_n)$, we have $\overline{J(H_1,\ldots,H_n)} \in (b _{i_0,1},\ldots,b _{i_k,1})$ where $1=i_0<i_1<\cdots<i_k$ are the indices of the first row in each Jordan block (i.e., $a _{i,i-1}=0$).  Therefore, not all $b _{i_j,1}$ are zero, so simply permute the Jordan blocks to obtain $b _{1,1} \neq 0$.
\end{proof}

We next define a special type of translation degenerate automorphism that appears in our factorization theorem.  These can be thought of as a generalization of Derksen's map (see Example \ref{ex:Derksen} below).

\begin{definition}\label{def:ttd}
A triangular automorphism $\tau\in\BA_n(\Ik)$ is said to be in {\em triangular translation degenerate form (TTD form)} 
if there exist $b_2,\ldots,b_n\in\Ik$ ($b_1=1$) and $d_2,\ldots,d_n\in\{0,1\}$ ($d_1=0$)
such that $(x_1)\tau=x_1$ and for each $2\le i\le n$:
$$(x_k)\tau=x_k+\sum_{r=1}^{k-1}\frac{(-1)^r x_1^r}{r!}\left(d_{k-r+1,k}\,x_{k-r}+d_{k-r+2,k}\,b_{k-r+1}\right)
+\frac{(-1)^k d_{2,k} x_1 ^k}{k!}.$$
where (for $2\leq k,j\leq n$):
 $$ d_{j,k}=\prod_{i=j}^k{d_i}=
 \begin{cases}       
  1           & {\rm if\ } 2\le k<j \\ 
d_j\cdots d_k & {\rm if\ } 2\le j\le k 
\end{cases}.$$
\end{definition}

\begin{example} \label{ex:Derksen}The triangular automorphism $(x_1,x_2-\frac{1}{2}x_1^2,x_3,\ldots,x_n)$, which is affinely equivalent to Derksen's map, is in TTD form with $d_2=1$, $d_3=\cdots=d_{n}=0$ ($d_1=0$), and $b_2=\cdots=b_n=0$ ($b_1=1$).
\end{example}

\begin{example}
The triangular automorphism in Example~\ref{ex:ttd} is in TTD form with $b_2=b_3=0$ ($b_1=1$) and $d_2=d_3=1$ ($d_1=0$).
\end{example}

\begin{remark}We make few remarks about this definition:
\begin{itemize}
\item Given a triangular automorphism $\tau\in\BA_n(\Ik)$ is in TTD we always  consider the parameters  
$b_2,\ldots,b_n\in\Ik$ ($b_1=1$) and $d_2,\ldots,d_n\in\{0,1\}$ ($d_1=0$)  as in Definition~\ref{def:ttd}
without explicitly mentioning it. 
\item A triangular automorphism $\tau \in \BA_n(\Ik)$ is in TTD form if and only if  
there exist $b_2,\ldots,b_n\in\Ik$ ($b_1=1$) and $d_2,\ldots,d_n\in\{0,1\}$ ($d_1=0$) such that $\tau=\nu\exp(-x_1D)$, 
where $D \in \LND _{\Ik[x_1]} \Ik^{[n]}$ is the triangular derivation given by $D(x_k)=d_kx_{k-1}+b_k$ and 
$\nu \in \BA_{n}(\Ik)$ is given by 
$(x_1)\nu=x_1$ and $(x_k)\nu = x_k + \frac{ (-1)^k}{k!}d_{2,k}\,x_1^k$ for all $2\le k\le n$. 
\item Given an integer $1\le k\le n$ and a triangular automorphism $\tau\in\BA_n(\Ik)$, 
$(x_k)\tau$ is affine if and only if $d_k=0$ and in this case $(x_k)\tau=x_k-b_kx_1$.
Thus $\tau$ is affine if and only if $d_k=0$ for all $1\le k\le n$.
\item As the name suggests, every triangular automorphism in TTD form is translation degenerate in $x_1$.  We don't actually require this fact, so we omit the proof which consists of a straightforward but somewhat tedious calculation.
\end{itemize}
\end{remark}

\begin{theorem}\label{thm:ttd}
Let $\phi \in \GA_n(\Ik)$ be an automorphism such that $\phi ^{-1} = (H_1,\ldots,H_n)$ and  $\frac{\partial  H_i}{\partial x_1 } = \sum _{j=1} ^n a _{i,j} H_j + b _{i}$ for some nilpotent lower triangular Jordan matrix  $(a _{i,j}) \in M_n(\Ik)$ and $b _{i} \in \Ik$ with $b _{1}=1$. 
Then there exists $\tau \in \BA_n(\Ik)$ in TTD form such that $\tau \phi ^{-1} = (H_1,G_2,\ldots,G_n)$ for some $G_i \in \Ik[x_2,\ldots,x_n]$.
\end{theorem}

\begin{proof} We consider $\phi \in \GA_n(\Ik)$ as in Theorem~\ref{thm:ttd}.
We set $d_i=a_{i,i-1} \in \{0,1\}$ for $2\le i\le n$ and $d_1=0$.  
Then by assumption, we have:
\begin{align*}
 \frac{\partial}{\partial x_1} H_k &=d_k  H_{k-1}+b_k
\end{align*}
for all $1\le k\le n$.
By induction, we deduce
\begin{align}\label{eq:H}
 \frac{\partial}{\partial x_1} \left( d_{k-r+1,k}H_{k-r}+d_{k-r+2,k}b_{k-r+1}\right) &=d_{k-r,k}H_{k-r-1}+d_{k-r+1,k}b_{k-r}
\end{align}
for all $1\le k\le n$ and $1\le r\le k-1$.

We remark that the convention $d_1=0$ gives $\frac{\partial}{\partial x_1} H_1=1$.
We consider $\tau\in\BA_n(\Ik)$ in TTD form with parameters $b_2,\ldots,b_n\in\Ik$ ($b_1=1$) 
and $d_2,\ldots,d_n\in\{0,1\}$ ($d_1=0$).
For any $1 \leq k \leq n$, we set $G_{k} =(x_{k}) \tau \phi ^{-1}$.  
Since $(x_1)\tau = x_1$, we have $(x_1)\tau \phi ^{-1} = H_1$.  
For each $2 \leq k \leq n$, we have:
$$G_k = H_k +\sum _{r=1} ^{k-1} \frac{(-1)^r H_1^r}{r!}\left( d_{k-r+1,k}H_{k-r}+d_{k-r+2,k}b_{k-r+1}\right)
+\frac{(-1)^k d_{2,k}H_1 ^k}{k!}.$$
Let $2 \leq k \leq n$. We compute $\frac{\partial }{\partial x_1}G_k$, first using the derivative of a product
and the formula~(\ref{eq:H}) and then changing the variable $r$ to $s=r-1$ in the first sum: 
\begin{align*}
\frac{\partial G_k}{\partial x_1} &=  d_k H_{k-1}+b_k+ \sum _{r=1} ^{k-1} \frac{(-1)^r H_1 ^{r-1}}{(r-1)!} \left( d_{k-r+1,k}H_{k-r}+d_{k-r+2,k}b_{k-r+1}\right)  + \\
&\phantom{xxx} \sum _{r=1} ^{k-1} \frac{(-1)^r H_1^r}{r!} (d_{k-r,k}H_{k-r-1}+d_{k-r+1,k}b_{k-r}) + \frac{(-1)^k d_{2,k} H_1 ^{k-1}}{(k-1)!} \\
&=  d_k H_{k-1}+b_k- \sum _{s=0} ^{k-2} \frac{(-1)^s H_1 ^s}{s!} \left( d_{k-s,k}H_{k-s-1}+d_{k-s+1,k}b_{k-s}\right)  + \\
&\phantom{xxx} \sum _{r=1} ^{k-1} \frac{(-1)^r H_1^r}{r!} (d_{k-r,k}H_{k-r-1}+d_{k-r+1,k}b_{k-r}) 
+ \frac{(-1)^k d_{2,k} H_1 ^{k-1}}{(k-1)!} \\
&= d_k H_{k-1}+b_k-d_k H_{k-1}-b_k + \frac{(-1)^{k-1} d_{2,k} H_1 ^{k-1}}{(k-1)!}+ \frac{(-1)^k d_{2,k} H_1 ^{k-1}}{(k-1)!} \\
&= 0.
\end{align*}
We deduce that $G_k \in \Ik[x_2,\ldots,x_n]$.
\end{proof}

We are finally ready to fully describe translation degenerate maps.

\begin{theorem} \label{thm:tdfactor}
Let $\phi \in \GA_n(\Ik)$ be translation degenerate.  
Then there exist $\lambda \in \GL_n(\Ik)$, $\rho \in \mathfrak{S}_n$, $\tau \in \BA_n(\Ik)$, 
and $G,G_2,\ldots,G_n \in \Ik[x_2,\ldots,x_n]$ such that setting 
\begin{align*}
\mu &= (x_1+G(x_2,\ldots,x_n),x_2,\ldots,x_n) \\
\gamma &= (x_1,G_2(x_2,\ldots,x_n),\ldots,G_n(x_2,\ldots,x_n) )
\end{align*} we have  $\phi ^{-1} = \lambda \tau ^{-1} \gamma \mu \rho$.  Moreover, $\tau $ is in TTD form.
\end{theorem}

\begin{proof}
First, note that an appropriate choice of $\rho \in \mathfrak{S}_n$ allows us to assume without loss of generality that $\phi$ is translation degenerate in $x_1$.
By Lemma \ref{lem:nilpotentJordan}, there exists $\lambda \in \GL_n(\Ik)$ such that $(\phi \lambda)^{-1} = (H_1,\ldots,H_n)$ and  $\frac{\partial  H_i}{\partial x_1 } = \sum _{j=1} ^n \alpha _{i,j} H_j + b _{i}$ where  $(a _{i,j}) \in M_n(\Ik)$ is a nilpotent Jordan matrix with $b _{1}\neq 0$.  Moreover, by altering $\lambda$ by a diagonal matrix, we may assume $b _1 = 1$.

By Theorem \ref{thm:ttd}, there exists $\tau \in \BA_n(\Ik)$ in TTD form such that  $\tau (\phi \lambda)^{-1} = (H_1,{G_2},\ldots,{G}_n)$ for some $G_i \in \Ik[x_2,\ldots,x_n]$..  Note that $\frac{\partial H_1}{\partial x_1}=1$, so $H_1=x_1+P(x_2,\ldots,x_n)$.  Then letting  $\mu = (x_1+P(x_2,\ldots,x_n),x_2,\ldots,x_n)$ and $\gamma = (x_1,{G_2},\ldots,{G}_n) \in \GA_n(\Ik)$, we have $\tau (\phi \lambda )^{-1} = \gamma \mu$, or $\phi ^{-1} = \lambda \tau ^{-1} \gamma \mu$ as required.
\end{proof}

With this classification in hand, we can now show that translation degenerate maps are either affine or \cotame.

\begin{theorem}\label{thm:tdcotame}
Let $\phi \in \GA_n(\Ik)$ be translation degenerate.  Then $\phi$ is either affine or \cotame.
\end{theorem}

\begin{remark}
The alert reader will note that Theorem \ref{thm:tdfactor} shows that every translation degenerate map is affinely equivalent to a biparabolic automorphism, and therefore is either affine or \cotame by a result of Bodnarchuk (Theorem \ref{thm:biparabolic}).  
However, by giving a direct proof of Theorem \ref{thm:tdcotame} here, 
we are able to use this in our new proof of Bodnarchuk's theorem in section \ref{sec:known}.
\end{remark}

\begin{proof}[Proof of Theorem~\ref{thm:tdcotame}]
By Theorem \ref{thm:tdfactor} and an affine equivalence, we may assume $\phi ^{-1} = \tau ^{-1} \pi \psi$ where $\tau$ is in TTD form, 
$\pi = (x_1\leftrightarrow x_n) \in \mathfrak{S}_n$, and $\psi \in \PA_n(\Ik)$.  Let $\alpha = (x_1,\ldots,x_{n-1},x_n+x_{n-1}) \in \Af_n(\Ik) \cap \BA_1(\Ik^{[n-1]})$.  Then
$$\phi ^{-1} \alpha \phi = \tau ^{-1} \pi \psi \alpha \psi ^{-1} \pi \tau =  \tau ^{-1} \pi \tilde{\alpha} \pi \tau$$
where $\tilde{\alpha}=\psi \alpha \psi ^{-1} = (x_1,\ldots,x_{n-1},x_n+(x_{n-1})\psi ^{-1}) \in \BA_1(\Ik^{[n-1]})$.  
Since showing that $\phi ^{-1} \alpha \phi$ is \cotame suffices to show $\phi$ is \cotame, we are thus reduced to showing that maps of the form $\tau ^{-1} \pi \tilde{\alpha} \pi \tau$ are nonaffine (Lemma \ref{lem:nonaffine}) and thus \cotame (Lemma \ref{lem:3trispecial}).
\end{proof}

\begin{lemma}\label{lem:nonaffine}
Let $\tau\in\BA_n(\Ik)$ be a nonaffine triangular automorphism in TTD form and 
let $\mu=(x_1+G,x_2,\ldots,x_{n-1},x_{n}+a)$ 
for some $G\in\Ik[x_2,\ldots,x_{n}]\pri\K$ and $a\in\K$.  
Then $\tau ^{-1}\mu\tau$ is nonaffine.
\end{lemma}

\begin{proof}
Since $\tau\theta_{n,a}=\theta_{n,a}\tau$, we can assume $a=0$.
For contradiction, we suppose $\tau^{-1}\mu \tau$ is affine.  
Set $\tilde{G}=(G)\tau$, and note that $\tilde{G}$ must be linear as $(x_1) \tau ^{-1} \mu \tau = x_1+\tilde{G}$. 

We distinguish two cases: first suppose that $d_2=1$, in which case $(x_2) \tau= x_2-b _2 x_1 - \frac{1}{2} x_1^2$ (we remind the reader the parameters $b_i$ and $d_i$ are from Definition \ref{def:ttd}).  
Then
$$(x_2) \tau ^{-1} \mu \tau = x_2+b _2 \tilde{G}+x_1\tilde{G}+\frac{1}{2} \tilde{G}^2.$$
Since we assumed $\tau ^{-1} \mu \tau$ is affine, we must have $\tilde{G}=-2x_1+c$ for some $c \in \Ik$, and thus $G=(\tilde{G})\tau ^{-1} = -2x_1+c$, 
contradicting $G \in \Ik[x_2,\ldots,x_{n}]$.

Now, we suppose instead that $d_2=0$. 
Since $\tau$ is nonaffine, we must have $d_r=0$ and $d_{r+1}=1$ for some $2\le r\le n-1$.  
Then we have
\begin{align*}
(x_r) \tau &= x_r-b _r x_1 \\
(x_{r+1})\tau &= x_{r+1}-x_1(x_{r}+b _{r+1})+ \frac{b _r}{2}x_1^2.
\end{align*}
Then 
\begin{align*}(x_{r+1}) \tau ^{-1} \mu \tau &= \left(x_{r+1}+(x_1+G)(x_r+b _{r+1})+ \frac{b _r}{2}(x_1+G)^2\right)\tau \\
&= x_{r+1}+\tilde{G}(x_r-b _r x_1+b _{r+1})+b _r x_1 \tilde{G} + \frac{b _r}{2} \tilde{G}^2.
\end{align*}
Thus we must have $b _r \neq 0$ and $\tilde{G}= -\frac{2}{b _r} x_r+c$ for some $c \in \Ik$, and thus $G=(\tilde{G})\tau ^{-1} = -\frac{2}{b _r}x_r-2x_1+c$. Again, this contradicts $G \in \Ik[x_2,\ldots,x_n]$, completing the proof.
\end{proof}

\begin{lemma}\label{lem:3trispecial}
Let $\tau\in\BA_n(\Ik)$ be triangular automorphism in TTD form and let $\mu=(x_1+G,x_2,\ldots,x_{n-1},x_{n}+a)$ 
for some $G\in\Ik[x_2,\ldots,x_{n}]$ and $a\in\K$.  
Then $\phi=\tau^{-1}\mu\tau$ is either affine or \cotame.
\end{lemma}

\begin{proof} Since $\tau\theta_{n,a}=\theta_{n,a}\tau$, we can assume $a=0$.
We first observe that if $\tau$ is affine then $\phi\simeq\mu\simeq\pi\mu\pi\in\PA_n(\Ik)$ 
where $\pi=(x_1\leftrightarrow x_n)$; and similarly, if $G \in \Ik[x_2,\ldots,x_{n-1}]$ then $\mu\in\PA_n(\Ik)$ 
and thus $\phi\in\PA_n(\Ik)$. In both cases, $\phi$ is affine or \cotame by Theorem~\ref{thm:parabolic}.

So we now assume that $\tau$ is nonaffine and $\deg_{x_n}G\ge 1$. 
We prove that $\phi$ is \cotame by induction on $\deg_{x_n}G$.
If $\deg_{x_n}G\ge 2$ then we set $\theta=\theta_{n,1}$ and compute
$$\phi^{-1}\theta\,\phi=\tau^{-1}\mu^{-1}\tau\theta\tau^{-1}\mu\tau
=\tau^{-1}\mu^{-1}\theta\mu\tau=\tau^{-1}\tilde{\mu}\tau$$
where $\tilde{\mu}=\mu^{-1}\theta\mu=(x_1+\tilde{G},x_2,\ldots,x_{n-1},x_n+1)$
and $\tilde{G}=-\Delta_n(G)\in\Ik[x_2,\ldots,x_{n}]$, 
so $\deg_{x_n}\tilde{G}=\deg_{x_n}G-1$ by Lemma~\ref{lem:partial}.  
Since $\phi^{-1}\theta_{n,1}\phi$ is \cotame implies $\phi$ is \cotame,
 we may assume we are in one of the subsequent cases with $\deg_{x_n}G=1$.

{\noindent \bf Case 1:} Suppose $G=Px_n+Q$ for some $P,Q\in\Ik[x_2,\ldots,x_{n-1}]$ with $P\notin\Ik$. 
As previously, (with again $\theta=\theta_{n,1}$) we have $\phi^{-1}\theta\phi=\tau^{-1}\tilde{\mu}\tau$
where $\tilde{\mu}=(x_1+\tilde{G},x_2,\ldots,x_{n-1},x_n+1)$, 
where $\tilde{G}=-\Delta_n(G)=-P\in\Ik[x_2,\ldots,x_{n-1}]\pri\K$.
Hence $\tilde{\mu}\in\PA_n(\K)$ and thus $\phi^{-1}\theta\phi=\tau^{-1}\tilde{\mu}\tau\in\PA_n(\K)$.
Since $\tau$ is nonaffine and $-P\not\in\K$, Lemma \ref{lem:nonaffine} implies that 
$\phi^{-1}\theta\phi=\tau^{-1}\tilde{\mu}\tau$ is nonaffine. Therefore $\phi^{-1}\theta\phi$
and thus $\phi$ are \cotame by Theorem~\ref{thm:parabolic}. 

{\noindent \bf Case 2:} Suppose $G=c x_n +Q$ for some $Q\in\Ik[x_2,\ldots,x_{n-1}]$ and $c\in\Ik^*$. 
In the following three subcases we consider a particular $\lambda\in\Af_n(\K)\cap\BA_n(\K)$
and set $\tilde{\lambda}:=\tau\lambda\tau^{-1}\in\Af_n(\K)$. 
We prove that $\tilde{\mu}:=\mu^{-1}\tilde{\lambda}\mu\in\PA_n(\K)$ and 
thus $\phi^{-1}\lambda\phi=\tau^{-1}\tilde{\mu}\tau\in\PA_n(\K)$.
We also prove that $\phi^{-1}\lambda\phi\not\in\Af_n(\K)$. Using Theorem~\ref{thm:parabolic}, we deduce that 
$\phi^{-1}\lambda\phi$ and thus $\phi$ is \cotame. 
In the first two subcases, the choice of $\lambda$ is easy; however, for the last case we require Lemma~\ref{lem:lambda} stated and proved below.

{\noindent \bf Case 2.1:} Suppose $d_i=0$ (hence $(x_i)\tau=x_i-b_ix_1$) for some $2\le i\le n-1$. 
We consider $\lambda = (x_1,\ldots,x_{n-1},x_n+x_i-b_ix_1)\in\Af_n(\K)\cap\BA_n(\K)$. 
Then $\tilde{\lambda}=(x_1,\ldots,x_{n-1},x_n+x_i)$ and 
$\tilde{\mu}=(x_1-cx_i,x_2,\ldots,x_n+x_i)\in\PA_n(\K)$.  
To check that $\phi ^{-1} \lambda \phi$ is nonaffine, note that 
$\phi^{-1}\lambda\phi\lambda ^{-1}=\tau^{-1}(x_1-cx_i,x_2,\ldots,x_n)\tau$ which is nonaffine by Lemma~\ref{lem:nonaffine}.

{\noindent \bf Case 2.2:} Suppose $d_{2,n-1}=1$ (equivalently $d_2=\cdots=d_{n-1}=1$) and $d_{n}=0$
(hence $(x_n)\tau=x_n-b_nx_1$). 
We consider $\lambda = (x_1,\ldots,x_{n-1},2x_n-(\frac{1}{c}+b_n)x_1)$. 
Then $\tilde{\lambda} = (x_1,\ldots,x_{n-1},2x_n-\frac{1}{c}x_1)$, 
and $\tilde{\mu} = (2x_1+Q,x_2,\ldots,x_{n-1},x_n-\frac{1}{c}x_1-\frac{1}{c}Q)$.
Since $Q\in\Ik[x_2,\ldots,x_{n-1}]$, we have $\tilde{\mu}\in\PA_n(\K)$.
Since $d_2=1$, to check that $\phi ^{-1} \lambda \phi$ is nonaffine, 
we proceed as in the proof of Lemma~\ref{lem:nonaffine}.
By contradiction, we suppose $\tau^{-1}\tilde{\mu}\tau$ is affine.  
We set $\tilde{Q}=(Q)\tau$ and note that $\tilde{Q}$ must be linear as 
$(x_1) \tau ^{-1} \tilde{\mu} \tau = 2x_1+\tilde{Q}$. 
We have $(x_2) \tau= x_2-b _2 x_1 - \frac{1}{2} x_1^2$.  
Then
$$(x_2)\tau^{-1}\tilde{\mu}\tau=x_2+b_2x_1+\frac{3}{2}x_1^2+(2x_1+\frac{b_2}{2}+\tilde{Q})\tilde{Q}$$
is affine. Thus $\tilde{Q}=-2x_1+c$ for some $c \in \Ik$, and thus $Q=(\tilde{Q})\tau ^{-1} = -2x_1+c$, 
contradicting $Q\in \Ik[x_2,\ldots,x_{n-1}]$.

{\noindent \bf Case 2.3:} Suppose $d_{2,n}=1$ (equivalently $d_2=\cdots=d_n=1$).  
Let $a \in \Ik^*$ be any non-root of unity, and set $g=c^{-1}(a-a^n) \neq 0$.  Let $\lambda\in\Af_n(\K)\cap\BA_n(\K)$ be as in Lemma~\ref{lem:lambda}, and set $\tilde{\lambda} = \tau \lambda \tau ^{-1}$.  
We compute
\begin{align*}
(x_1)\tilde{\mu} &:= (x_1)\mu^{-1}\tilde{\lambda}\mu=\left(x_1-cx_n-Q \right) \tilde{\lambda}  \mu \\
&=\left(ax_1 -c\left(a^nx_n+gx_1+(a^n-a^{n-1}) \sum _{r=2} ^{n-1} w _{k-r} x_r \right)-(Q)\tilde{\lambda} \right) \mu \\
&=\left(a^nx_1 -c\left(a^nx_n+(a^n-a^{n-1}) \sum _{r=2} ^{n-1} w _{k-r} x_r \right)-(Q)\tilde{\lambda} \right) \mu \\
&= a^n(x_1+cx_n+Q) -c\left(a^nx_n+(a^n-a^{n-1}) \sum _{r=2} ^{n-1} w _{k-r} x_r \right)-(Q)\tilde{\lambda} \\
&= a^nx_1+R,
\end{align*}
where $R=a^nQ-c(a^n-a^{n-1}) \sum _{r=2} ^{n-1} w _{k-r} x_r -(Q)\tilde{\lambda}\in \K[x_2,\ldots,x_{n-1}].$

Note that, for each $2 \leq k \leq n-1$, we have $(x_k)\tilde{\lambda} \in \K[x_2,\ldots,x_k]$, 
so we have $(x_k) \tilde{\mu}=(x_k)\tilde{\lambda}$. We deduce $\tilde{\mu} \in \PA_n(\K)$ and thus
$\phi ^{-1} \lambda \sigma \phi = \tau ^{-1} \tilde{\mu} \tau \in\PA_n(\K).$
So we are left to check that $\tau ^{-1} \tilde{\mu} \tau$ is not affine.

First, suppose $R \notin \K$, in which case $(R)\tau$ must have degree at least $2$. 
Since $(x_1) \tau ^{-1} \tilde{\mu} \tau = ax_1+(R)\tau$ we deduce $\tau ^{-1} \tilde{\mu} \tau$ is not affine, as required.

We may now assume that $R \in \K$. 
We compute
\begin{align*}
(x_2) \tau ^{-1} \tilde{\mu}\tau &= \left(x_2+\frac{1}{2}x_1^2+b_2x_1\right) \tilde{\mu} \tau \\
&= \left( a^2x_2+(a^2-1)b_3 + \frac{1}{2}(a^nx_1+R)^2+b_2(a^nx_1+R) \right) \tau \\
&=a^2 x_2 + \frac{1}{2}(a^{2n}-a^2)x_1^2+\left( b_2(a^n-a^2)+a^nR\right)x_1+\tilde{R}
\end{align*}
for some $\tilde{R} \in \K$.  Since $a$ is not a root of unity, we see $\tau ^{-1} \tilde{\mu} \tau$ is not affine, as required.

\end{proof}

\begin{lemma}\label{lem:lambda}
Let $\tau \in \BA_n(\K)$ be a triangular automorphism in TTD form with $d_{2,n}=1$, and let $g \in \K$ and $a \in \K^*$.  
There exists $\lambda \in \Af_n(\K) \cap \BA_n(\K)$ and $w_1,\ldots,w_{n-2}\in \K$ 
such that, setting  $\tilde{\lambda}=\tau\lambda\tau^{-1}$,
\begin{align*}
(x_k) \tilde{\lambda} &= \begin{cases}
ax_1 & k=1 \\
a^kx_k + (a^k-1) b _{k+1} +(a^k-a^{k-1}) \sum _{r=2} ^{k-1} w _{k-r}x_r & 2 \leq k \leq n-1 \\
a^n x_n+ gx_1 + (a^n-a^{n-1}) \sum _{r=2} ^{n-1} w_{n-r}x_r & k=n
\end{cases}.
\end{align*}
Moreover, the $w_j$ satisfy the recursive definition
\begin{equation}\label{eq:w}
w _j = \begin{cases} 
-b _2                                & {\rm if\ } j=1 \\ 
-\sum _{i=1} ^{j-1} w _i b _{j-i+1} & {\rm if\ } 2\le j\le n 
\end{cases}.
\end{equation} 
\end{lemma}

\begin{proof}
Let $\lambda_1 \in \GL_n(\K) \cap \BA_n(\K)$ be given by 
$$(x_k) \lambda_1 = x_k+\frac{a-1}{a} \sum _{r=1} ^{k-1} w_{k-r}x_r$$
where $w_j$ are defined as in \eqref{eq:w}.  The proof is a straightforward but moderately unpleasant computation; for each $2 \leq k \leq n$, we have

\begin{align*}
(x_k) \tau \lambda_1 \tau ^{-1} 
&= \left( x_k+\sum _{i=1} ^{k-1} \frac{ (-1)^i x_1 ^i }{i!}(x_{k-i}+b_{k-i+1}) 
+ \frac{(-1)^k x_1 ^k}{k!} \right) \lambda_1 \tau ^{-1}
\end{align*}
(by definition of $\tau$)
\begin{align*}
&= \left( x_k+\frac{a-1}{a} \sum _{r=1} ^{k-1} w_{k-r}x_r + \sum _{i=1} ^{k-1} \frac{ (-1)^i x_1 ^i }{i!}\left(x_{k-i} + \frac{a-1}{a} \sum _{s=1} ^{k-i-1} w_{k-i-s}x_s+b_{k-i+1}\right) \right.\\
&\phantom{xxx} \left.  + \frac{(-1)^k x_1 ^k}{k!} \right) \tau ^{-1}
\intertext{(by definition of $\lambda_1$)}
&= \left( x_k + \sum _{i=1} ^{k-1} \frac{ (-1)^i x_1 ^i }{i!}(x_{k-i}+b_{k-i+1}) + \frac{(-1)^k x_1 ^k}{k!}\right. \\
&\phantom{xxx} \left. +\frac{a-1}{a} \sum _{r=1} ^{k-1} w_{k-r}x_r+\sum _{i=1} ^{k-2} \frac{(-1)^i x_1^i}{i!} \left( \frac{a-1}{a} \sum _{s=1} ^{k-i-1} w_{k-i-s}x_s\right) \right) \tau ^{-1}
\intertext{(because the sum $\sum _{s=1} ^{k-i-1}$ is equal to $0$ when $i=k-1$)}
&= x_k + \frac{a-1}{a}\left( \sum _{r=1} ^{k-1} w_{k-r} x_r + \sum _{i=1} ^{k-2} \frac{(-1)^i}{i!} x_1^i \sum _{s=1} ^{k-i-1} w_{k-i-s} x_s \right) \tau ^{-1} 
\intertext{(by definition of $\tau$)}
&= x_k + \frac{a-1}{a}\left( \sum _{r=1} ^{k-1} w_{k-r} x_r + \sum _{r=2} ^{k-1} w_{k-r}\sum _{i=1} ^{r-1} \frac{(-1)^i}{i!} x_1^i  x_{r-i} \right) \tau ^{-1} 
\intertext{(in the sum $\sum _{s=1} ^{k-i-1}$ we changed the variable $s$ to $r=s+i$ and then permuted 
the two sums: $\sum_{i=1}^{k-2}\sum_{r=i+1}^{k-1}=\sum_{r=2}^{k-1}\sum_{i=1}^{r-1}$)}
&= x_k+\frac{a-1}{a} \left( w_{k-1}x_1+\sum _{r=2} ^{k-1} w_{k-r}\left(x_r- \sum _{i=1} ^{r-1} \frac{(-1)^i b_{r-i+1}}{i!} x_1^i -\frac{(-1)^r}{r!}x_1^r\right)\right)
\intertext{(we separed the first term of the first sum, factor $\sum _{r=2} ^{k-1} w_{k-r}$ and then used the definition of $\tau$)}
&= x_k + \frac{a-1}{a} \left( w_{k-1}x_1+\sum _{r=2} ^{k-1} w_{k-r}\left(x_r -\frac{(-1)^r}{r!}x_1^r\right)  - \sum _{i=1} ^{k-2} \frac{(-1)^i}{i!}x_1^i \sum _{r=i+1} ^{k-1} w_{k-r}b_{r-i+1} \right) 
\intertext{(we permuted the two sums: $\sum_{r=2}^{k-1}\sum_{i=1}^{r-1}=\sum_{i=1}^{k-2}\sum_{r=i+1}^{k-1}$, note that this is the converse as previously)}
&= x_k + \frac{a-1}{a} \left(\sum _{r=2} ^{k-1} w_{k-r}x_r - \frac{(-1)^{k-1}w_1}{(k-1)!} x_1 ^{k-1} \right).
\end{align*}
(changing the variable $r$ to $s=k-r$, we have: 
$\sum_{r=i+1}^{k-1} w_{k-r}b_{r-i+1}=\sum_{s=1}^{k-i-1} w_{s}b_{k-i-s+1}=-w_{k-i}$ )

Now we set $b_{n+1}=0$ for convenience, and let $\lambda _0 \in \Af_n(\K) \cap \BA_n(\K)$ be given by
$$(x_k) \lambda _0 = \begin{cases} 
ax_1                     & {\rm if\ }k=1 \\ 
a^kx_k+(a^k-1)b _{k+1}   & {\rm if\ } 2\le k\le n 
\end{cases}.$$ 
Then we have $(x_1) \tau \lambda _0 \tau ^{-1}=ax_1$ and, for $2\le k\leq n$, we observe
\begin{align*}
(x_k) \tau  &= x_k+\sum_{i=1}^{k-2}\frac{(-1)^i x_1^i}{i!}(x_{k-i}+b_{k-i+1})
+\frac{(-1)^{k-1} x_1^{k-1}}{(k-1)!}b_2+\frac{(-1)^k x_1 ^k}{k!}(1-k)
\end{align*}
(we separated the term $k-1$ in the sum)
\begin{align*}
(x_k) \tau\lambda_0  &= a^k\left(x_k+\sum_{i=1}^{k-2}\frac{(-1)^i x_1^i}{i!}(x_{k-i}+b_{k-i+1})
+\frac{(-1)^k x_1^k}{k!}(1-k)\right.\\
                     & \phantom{xxx}\left.+a^{-1}\frac{(-1)^{k-1} x_1^{k-1}}{(k-1)!}b_2\right)+(a^k-1)b_{k+1}
\end{align*}
(by definition of $\lambda_0$, two $b_{k-i+1}$ terms canceled)
\begin{align*}
(x_k) \tau \lambda_0 \tau ^{-1} &=a^k\left(x_k-\frac{a-1}{a}\frac{(-1)^{k-1}b_2}{(k-1)!}x_1^{k-1}\right)+(a^k-1)b_{k+1}.
\end{align*}
Finally, define $\lambda _2 \in \Af_n(\Ik) \cap \BA_1(\Ik)$ by $\lambda _2 = (x_1,\ldots,x_{n-1},x_n+\frac{g}{a^n}x_1)$.  Then since $\lambda _2$ and $\tau$ commute and $w_1=-b_2$, it follows that setting $\lambda = \lambda_0\lambda_1 \lambda _2$ satisfies the lemma.
\end{proof}

\section{A $4$-triangular automorphism that is not \cotame}\label{s:lfour}

In this section, we improve the authors' result of \cite{EL} to produce an example of a $4$-triangular automorphism that is not \cotame.  This section can be read independently of the previous sections, but the reader will want to be familiar with the techniques used in \cite{EL}; for the sake of brevity, we adopt all the notations therein for the remainder of this paper, and  restrict our attention to $n=3$. In particular, we set
$\beta=(x+y^2(y+z^2)^2,y+z^2,z)\in \BA_3(\Ik)$, $\pi=(y,x,z)\in \Af_3(\Ik)$ and
$\theta_N=(\pi\beta)^N\pi(\pi\beta)^{-N}$.
The goal of this section is to prove:

\begin{theorem}[cf. \cite{EL} Theorem B]\label{thmb}
The automorphism $\theta_2$ is not \cotame.
\end{theorem}
  
The idea in \cite{EL} is to track the growth of particular sets of polynomials satisfying certain degree constraints after repeated applications of the maps $\pi \beta$, $\pi \beta ^{-1}$, and affine maps.  The crucial technical theorem is

\begin{theorem}[\cite{EL} Theorem 3]\label{thm3} If $N \geq 3$, then the set $\mathcal{P}^*$ is stable under the action of the automorphisms $\pi \beta$, $\pi \beta ^{-1}$, and $(\pi \beta ^{-1})^N \alpha \pi (\pi \beta)^N$ for any $\alpha \in \mathcal{A} \pri \mathcal{A}_4$.
\end{theorem}

However, this fails for $N=2$.  We thus modify the argument by considering the sets

$$\mathcal{P}_4^*=\bigcup_{m\ge 4,n\ge 0}\mathcal{P}_{m,n} ^* \subset \mathcal{P}^*{\ \ \rm and\ \ }\mathcal{Q}_4^*=\bigcup_{m\ge 4,n\ge 0}\mathcal{Q}_{m,n} ^* \subset \mathcal{Q}^*.$$
Note that for all integers $m\ge 1$ and $n\ge 0$, $\mathcal{P}_{m,n}^*\subset \mathcal{Q}_{4m,n}^*$, so we have $\mathcal{P}^*\subset \mathcal{Q}_4^*$.

We prove an analagous theorem that holds for all $N \geq 2$.

\begin{theorem} If $N \geq 2$, then the set ${\cal Q}_4^*$ is stable under the action of the automorphisms $\pi\beta$, $\pi\beta^{-1}$
and $(\pi\beta^{-1})^N\al\pi(\pi\beta)^N$ for any $\al\in{\cal A}\pri{\cal A}_4$.
\end{theorem}

Before giving the proof, we note that we immediately obtain the following corollaries analagous to \cite{EL}.

\begin{corollary}[cf. \cite{EL} Corollary 4]
Let $r\ge 1$ be an integer. Let $\al_0,\ldots,\al_{r}\in{\cal A}$, and set $\phi=\al_0\theta_2\al_1\cdots\theta_2\al_{r}$.
If $\al_1,\ldots,\al_{r-1}\in{\cal A}\pri {\cal A}_4$, then there exist $\al,\al'\in {\cal A}$ such that
$(y)\al\phi\al'\in{\cal Q}_4^*$.
\end{corollary}

\begin{corollary}[cf. \cite{EL} Corollary 5] Let $\phi\in\langle{\cal A},\theta_2\rangle\pri{\cal A}$.  
Then there exist $\al,\al'\in {\cal A}$ such that $(y)\al\phi\al'\in{\cal Q}_4^*$.
\end{corollary}

\begin{corollary}[cf. \cite{EL} Corollary 6]
We have: ${\cal C}={\cal A}_4$.  In particular, ${\mathcal C}$ is a finite cyclic group of order $1,2,3$ or $6$.
\end{corollary}

If $f\in{\cal Q}_4^*$ then ${\rm ldeg}_2(f)=(0,m,n)$ and $\deg_{(1,1,1)}(f)\ge 4$, so we obtain the following version of 
Corollary~7.

\begin{corollary}[cf. \cite{EL} Corollary 7] 
Let $\phi\in \TA_3(\Ik)$ be a tame automorphism with $\deg_{(1,1,1)}((f)\phi)\le 4$ for all $f \in \Ik[\bx]$ with $\deg _{(1,1,1)} (f) = 1$.  If $\phi \notin \mathcal{A}$, then $\phi\not\in\langle{\cal A},\theta_2\rangle$.  In particular, $ \langle \mathcal{A}, \theta_2 \rangle$ is a proper subgroup of $\mathcal{T}$.
\end{corollary}

We observe that this, together with Theorem \ref{mt1}, gives Theorem \ref{mt2}.

\begin{corollary}[cf. \cite{EL} Corollary 8] 
The group $\langle{\cal A},\theta_2\rangle$ is the amalgamated free product of ${\cal A}$ and
$\langle{\cal C},\theta_2\rangle$ along their intersection ${\cal C}$.
\end{corollary}

\begin{remark}
Using that that ${\cal C}$ is a finite group, one can easily check 
that the group $\langle{\cal A},\theta_N\rangle$ (for all $N\ge 2$) 
shares with $\TA_2(\K)$ the property of being acylindrically hyperbolic (see~\cite{Minasyan} for the definition).
\end{remark}

We conclude the paper by proving Theorem \ref{thm3}.
\begin{proof}[Proof of Theorem \ref{thm3}]
If $\gamma\in\{\beta,\beta^{-1}\}$ then $({\cal Q}_4^*)\pi\gamma\subset{\cal Q}_4^*$ by Lemma~9 of~\cite{EL}.
If $\al\in{\cal A}\pri{\cal A}_3$ then $({\cal Q}_4^*)(\pi\beta^{-1})^2\al\beta\pi\beta\subset{\cal Q}_4^*$
by Propositions~12, 14 and 15 of \cite{EL}. It remains to prove that $({\cal Q}_4^*)(\pi\beta^{-1})^2\al\beta\pi\beta\subset{\cal Q}_4^*$
in the case $\al\in{\cal A}_3\pri{\cal A}_4$.
In this case, we can write $\al=(u^8x+cz+d,u^2y,uz)$ for some $u\in\Ik^*$ and $c,d\in\Ik$).
Set $\gamma=\beta^{-1}\pi\beta^{-1}\al\beta\pi\beta$ and compute
$$\gamma=(z^2x+u^2y^2(y+z^2)^2-(u^8y+Z_1)^2(u^8y+Z_2)^2,u^8y+Z_2,uz)$$
where $Z_1=u^8z^2+cz+d$ and $Z_2=(u^8-u^2)z^2+cz+d$ are polynomials in $z$ of degree $\le 2$. Moreover 
the degree of $Z_2$ is $\le 1$ when $u^6=1$.

We set $X=(x)\gamma$, $Y=(y)\gamma$, and $Z=(z)\gamma$, and examine their degrees.  Expanding $X$ we have:
$$X=z^2x+u^2(1-u^{30})y^4+2u^2((1-u^{24}(2u^6-1))z^2-u^{22}cz-u^{22}d)y^3$$
$$\hspace*{1cm}+(z^4-u^{16}(Z_1^2+4Z_1Z_2+Z_2^2))y^2-2u^2(Z_1Z_2^2+Z_1^2Z_2)y-Z_1^2Z_2^2.$$

To compute the relevant degrees of $X$, we must consider $4$ individual cases.
$$
\begin{array}{|c|c|c|c|}
  \hline
  {\rm Case\ }A &  {\rm Case\ }B(2) & {\rm Case\ }B(1) & {\rm Case\ }B(0)  \\
  \hline
  u^{30}\ne 1 & u^{30}=1{\rm\ and\ }u^6\ne 1 & u^6=1{\rm\ and\ }c\ne 0 & u^6=1,\,c=0{\rm\ and\ }c\ne 0\\
  \hline
 
\end{array}
$$

In each of these cases, we summarize the relevant degrees of $X$,$Y$, and $Z$.  In the table, $l \in \{0,1,2\}$ distinguishes between cases B(0), B(1), and B(2).

$$
\begin{array}{|c|c|c|c|c|c|c|c|c|}
  \hline
   & \deg_{(1,1,0)} & \deg_{(3,3,1)} & {\rm ldeg}_2  \\
  \hline
  X\, ({\rm Case\ } A)  & 4    & 12 & (0,4,0) \\
   \hline
  X\, ({\rm Case\ } B(l)) &  3 & 9+l & (0,3,l) \\
  \hline
  Y & 1                 & 3                             & (0,1,0)  \\
  \hline
  Z & 0                 & 1                             & (0,0,1)  \\
  \hline
\end{array}
$$
Now, we prove that $({\cal Q}_4^*)\pi\gamma\subset{\cal Q}_4^*$ with the same technique as in Lemma~9 of~\cite{EL}.
Let $m\ge 4$ and let $n\ge 0$ be integers, and let $P\in\mathcal{Q}_{m,n}^{*}$. For all $v=(i,j,k)\in\supp(P)$, we have $i+j\le m$, $3i+3j+k\le 3m+n$
and $(\bx^v)\pi\gamma=Y^{i}X^{j}Z^{k}$.\\

First, in case $A$ we compute 
\begin{align*}
\deg_{(1,1,0)}((\bx^v)\pi\gamma)&=i+4j\le 4(i+j)\le 4m,\\
\deg_{(3,3,1)}((\bx^v)\pi\gamma)&=3i+12j+k\le 9(i+j)+3i+3j+k\\
                                &\le 9m+3m+n=3(4m)+n,\\
{\rm ldeg}_2((\bx^v)\pi\gamma)&=(0,i+4j,k)\le_2 (0,4m,n).
\end{align*}
We check that the last inequality is an equality if and only if $(i,j,k)=(0,m,n)$ which belongs to $\supp(P)$; thus $({\cal Q}_{m,n}^*)\pi\gamma\subset{\cal Q}_{4m,n}^*$.\\

Now in cases $B(l)$ we compute 
\begin{align*}
\deg_{(1,1,0)}((\bx^v)\pi\gamma)&=i+3j\le 3(i+j)\le 3m,\\
\deg_{(3,3,1)}((\bx^v)\pi\gamma)&=3i+(9+l)j+k\le (6+l)(i+j)+3i+3j+k\\
                                &\le (6+l)m+3m+n=3(3m)+lm+n,\\
{\rm ldeg}_2((\bx^v)\pi\gamma)&=(0,i+3j,lj+k)\le_2 (0,3m,lm+n).
\end{align*}
We check that the last inequality is an equality if and only if $(i,j,k)=(0,m,n)$ which belongs to $\supp(P)$; thus $({\cal Q}_{m,n}^*)\pi\gamma\subset{\cal Q}_{3m,lm+n}^*$.
\end{proof}

\end{document}